\documentclass[11pt]{article}
\usepackage{amssymb,amsfonts,amsmath,latexsym,epsf,tikz,url}
\usepackage[usenames,dvipsnames]{pstricks}
\usepackage{pstricks-add}
\usepackage{epsfig}
\usepackage{pst-grad} 
\usepackage{pst-plot} 
\usepackage[space]{grffile} 
\usepackage{etoolbox} 
\usepackage{float}
\makeatletter 
\patchcmd\Gread@eps{\@inputcheck#1 }{\@inputcheck"#1"\relax}{}{}
\makeatother

\newtheorem{theorem}{Theorem}[section]

\newtheorem{problem}[theorem]{Problem}

\newcommand{\qed}{\hfill $\square$\medskip}

\begin{document}

\def\nt{\noindent}

\title{ Coalition  of cubic graphs of order at most $10$ }

\author{\small 	
Saeid Alikhani$^{1}$ 
\and
\small Hamid Reza Golmohammadi$^{2}$  
\and
\small Elena V. Konstantinova$^{2,3}$
}


\maketitle

\begin{center}

$^1$Department of Mathematical Sciences, Yazd University, 89195-741, Yazd, Iran\\

$^{2}$Novosibirsk State University, Pirogova str. 2, Novosibirsk, 630090, Russia\\ 

$^3$Sobolev Institute of Mathematics, Ak. Koptyug av. 4, Novosibirsk,
630090, Russia\\
\medskip
{\tt alikhani@yazd.ac.ir ~~ h.golmohammadi@g.nsu.ru~~~e\_konsta@math.nsc.ru}
\end{center}

\begin{abstract}
    The  coalition in a graph $G$  consists of two disjoint sets of vertices $V_{1}$ and $V_{2}$, neither of which is a dominating set but whose union $V_{1}\cup  V_{2}$, is a dominating set. A coalition partition in a
 graph $G$  is a vertex partition $\pi$ = $\{V_1, V_2,..., V_k \}$ such that every set $V_i \in \pi$ is not a dominating set but
 forms a coalition with another set $V_j\in \pi$ which is not a dominating set. 
 The coalition number $C(G)$
 equals the maximum  $k$ of a coalition partition  of $G$. 
  In this paper, we compute the coalition number of all cubic graphs of order at most $10$.
\end{abstract}


\medskip
\noindent{\bf AMS Subj.\ Class.:}  05C60.


\section{Introduction} 
Let $G = (V,E)$ be
a simple graph of order $n$. The open neighborhood (closed neighborhood) of a vertex
$v \in V$ is the set $N(v)$ = $\{u | uv \in E\}$, (the set
$N[v]$ = $N(v) \cup \{v\}$). The number of vertices in  $N(v)$
is the degree of $v$, denoted by $ deg(v)$. A vertex of
degree $n-1$ is called a full or universal vertex, while a vertex of degree $0$ is an
isolate one. As usual the minimum degree and the maximum degree of $G$ is denoted by $\delta(G)$ and $\Delta(G)$, respectively. A subset $V_{i} \subseteq V$ is called a singleton set if 
$|V_i| = 1$
and  a non-singleton
set, if $|V_i|\ge 2$.  A set $S\subseteq V$ is a dominating set of  $G$, if every vertex in $V \setminus S$
has at least one neighbor in $S$.
A set $S \subseteq V$ is a total dominating set of  $G$ with no
isolated vertex, if every vertex in
$V$  has at least one neighbour in $S$; in other words $N(S)=V$. The literature on this subject has been
surveyed and detailed in the two excellent so-called domination books by Haynes,
Hedetniemi, and Slater \cite{7,8}. A domatic partition is a partition
of the vertex set into dominating sets. Formally, the domatic number
$d(G)$ equals the maximum order $k$ of a vertex partition, called a domatic partition, $\pi$ = $\{V_1, V_2, . . . , V_k \}$
such that every set $V_{i}$ is a dominating set in $G$. The domatic number of a graph was introduced by Cockayne and
 Hedetniemi \cite{2}. For more details on the domatic number refer to e.g., \cite{11,12,13}.

 Coalitions and coalition partitions were introduced in \cite{4},
and are  now studied in graph theory (see for example \cite{5,6}). The concepts were defined  in terms of general graph properties but focused
on the property of being a dominating set. 
 A coalition $\pi$ in a graph $G$ consists of two disjoint sets of vertices
 $V_1$ and $V_2$, neither of which is a dominating set of $G$ but whose union $V_1 \cup V_2$ is
 a dominating set of $G$. We say that the sets $V_1$ and $V_2$ form a coalition and that
 they are coalition partners in $\pi$.

A coalition partition, henceforth called a c-partition, in a graph
$G$ is a vertex partition $\pi$ = $\{V_1, V_2, \ldots, V_k \}$ such that every set $V_i$ of $\pi$ is either
a singleton dominating set of $G$, or is not a dominating set of $G$ but forms a coalition with another non-dominating set $V_j\in \pi$. The coalition number $C(G)$
equals the maximum order $k$ of a $c$-partition of $G$. 
Associated with every coalition partition $\pi$ of a graph $G$ is
a graph called the coalition graph of $G$ with respect to $\pi$, denoted $CG(G,\pi)$,
the vertices of which correspond one-to-one with the sets $V_1, V_2,...,V_k$ of $\pi$
and two vertices are adjacent in $CG(G,\pi)$ if and only if their corresponding
sets in $\pi$ form a coalition (\cite{5}). In  \cite{5} the coalition graphs (focusing
on the coalition graphs of paths, cycles, and trees) have studied. 
Recently Bakhshesh, Henning and Pradhan in \cite{Davood} characterized all graphs $G$ of order $n$ with $\delta(G)\leq 1$ and $C(G)=n$, and also characterized all trees $T$ of order $n$ with $C(T)=n-1$. 

The class of cubic graphs
is especially interesting for mathematical applications, because for various important open problems in
graph theory, cubic graphs are the smallest or simplest possible potential counterexamples, and so this creates motivation to study coalition parameter for the cubic graphs of order at most $10$.

  Alikhani and Peng have studied the domination polynomials (which is the generating function for the number of dominating sets of a graph) of cubic graphs of order $10$ in \cite{1}. As a
  consequence, they have shown that the Petersen graph is determined uniquely by its domination polynomial. 

\medskip
In the next section, we compute the coalition number of cubic graphs of order $6$ and $8$. 
Also, we compute the coalition number of the cubic graphs of order $10$ in Section 3.

\section{Coalition number of cubic graphs of order $6$ and $8$ }

In this section, we  study the coalition number  of the cubic graphs of order $6$ and $8$. 
We need  the following theorems. 

\begin{theorem}{\rm \cite{4}} \label{order}
If $G$ is a graph of order $n$, then $1\leq C(G)\leq n$. 
\end{theorem}

\begin{theorem}{\rm \cite{4}} \label{low}
	If $G$ is a graph with no full vertex and minimum
	degree $\delta(G)\ge 1$, then  $C(G)\ge\delta(G)+2$.
\end{theorem}

\begin{theorem} {\rm \cite{6}}\label{up}
	For any graph $G$, $C(G)\leq\frac{(\Delta(G)+3)^2}{4}$.
\end{theorem}

\begin{theorem} {\rm \cite{4}} \label{atmost} 
	Let $G$ be a graph with maximum degree $\Delta(G)$,
	and let $\pi$ be a $C(G)$-partition. If $X \in \pi$, then $X$ is in at most $\Delta(G)+1$ coalitions.
\end{theorem}

First we  determine the coalition number of the cubic graphs of order $6$. There are exactly two  cubic graphs of order $6$ which are denoted by $G_{1}$ and $G_{2}$ in Figure \ref{Cubic6}.
\begin{figure}[H]
	\hglue2.5cm
	\includegraphics[width=9cm,height=3.3cm]{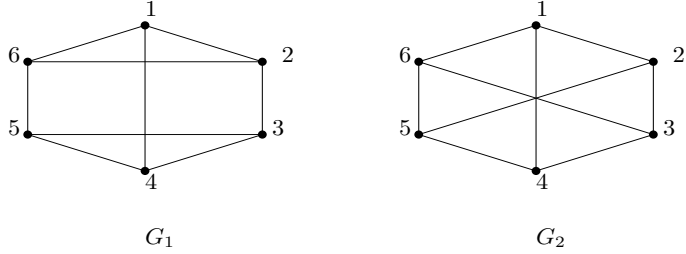}
	\vglue-10pt \caption{\label{Cubic6} Cubic graphs of order $6$.}
\end{figure}

\begin{theorem}\label{8}
	The coalition number of the cubic graphs $G_{1}$ and $G_{2}$  (Figure \ref{Cubic6}) of order $6$ is $6$.  
\end{theorem} 
\begin{proof} 
	Theorems \ref{order} and \ref{low}  show that $5\leq C(G)\leq 6$. Let compute $C(G)$ for the graph $G_1$. We create a partition of order $6$. Let $\pi= \big\{V_{1}=\{1,4\}, V_{2}= \{2,5\}, V_{3}= \{3,6\}\big\}$ be a domatic partition of a graph $G_{1}$, where $d(G_{1})=3$. Since any partition of a non-singleton, minimal dominating set into two nonempty sets creates two non-dominating sets whose union forms a coalition, so we can partition each of the minimal dominating sets $V_{1}=\{1,4\}$, $V_{2}= \{2,5\}$ and $V_{3}= \{3,6\}$ into two sets such as $V_{1,1}=\{1\}$, $V_{1,2}=\{4\}$, $V_{2,1}=\{2\}$, $V_{2,2}=\{5\}$, $V_{3,1}=\{3\}$ and $V_{3,2}=\{6\}$, which each of $V_{1,1}$, $V_{2,1}$ and $V_{3,1}$ form a coalition with each of $V_{1,2}$, $V_{2,2}$ and $V_{3,2}$, respectively. Therefore, we have a $c$-partition of $G_{1}$ of order $6$ as follows:
	
	$\pi_{1} =\big \{V_{1,1}=\{1\}, V_{1,2}=\{4\}, V_{2,1}=\{2\}, V_{2,2}=\{5\}, V_{3,1}=\{3\}, V_{3,2}=\{6\} \big\}$.	
	\qed
\end{proof}

\medskip

Now, we compute the  coalition number of cubic graphs of order $8$. There are exactly $6$ cubic graphs of order $8$ which is denoted by  $G_{1},G_{2},...,G_{6}$ in Figure \ref{Cubic8}. 

\begin{figure}[H]
	\hglue2.5cm
	\includegraphics[width=11cm,height=7.2cm]{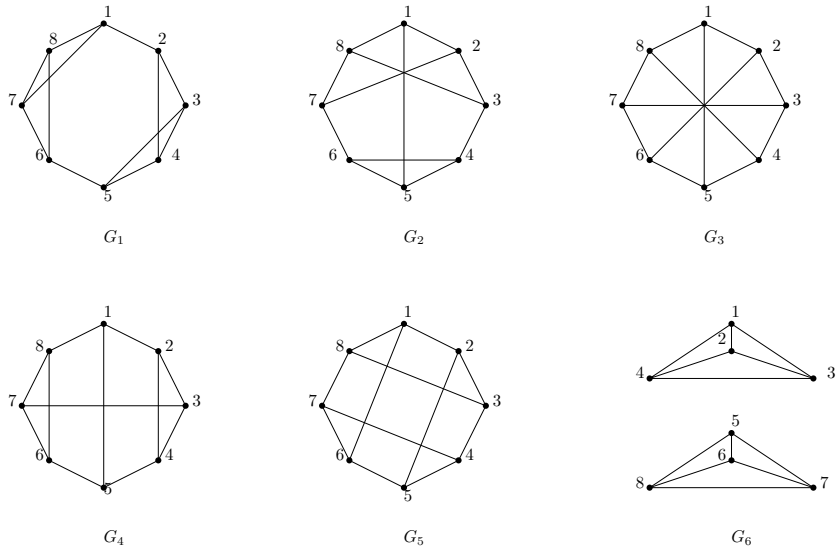}
	\vglue-10pt \caption{\label{Cubic8} Cubic graphs of order $8$.}
\end{figure}

The following theorem gives the coalition numbers of cubic graphs of order $8$: 

\begin{theorem}\label{8}
	For the cubic graphs $G_{1},G_{2},...,G_{6}$ of order $8$ (Figure \ref{Cubic8}) we have: 	
	\begin{enumerate} 
		\item[(i)] 
		$C(G_1)=C(G_5)=C(G_6)=8$. 
		\item[(ii)] 
		$C(G_4)=7.$  
		\item[(iii)]  
		$C(G_2)=C(G_3)=6$. 	 
	\end{enumerate} 
\end{theorem} 
\begin{proof}
	\begin{enumerate}
		\item[(i)] 
		Using Theorems \ref{order} and \ref{low}, we have $5\leq C(G)\leq 8$. Without loss of generality, we compute $C(G)$ for the graph $G_1$ (with similar argument we can determine the coalition number for other graphs in this part).
		We show that there is a partition of order $8$ for the graph $G_{1}$.  Let $\pi= \big\{V_{1}=\{1,5\}, V_{2}= \{2,6\}, V_{3}= \{3,8\}, V_{4}= \{4,7\} \big\}$ be a domatic partition of the graph $G_{1}$, where $d(G_{1})=4$. Since any partition of a non-singleton, minimal dominating set into two nonempty sets creates two non-dominating sets whose union forms a coalition, so we can partition each of the minimal dominating sets in $\pi$ into two sets such as $V_{1,1}=\{1\}, V_{1,2}=\{5\}, V_{2,1}=\{2\}, V_{2,2}=\{6\}$, $V_{3,1}=\{3\}, V_{3,2}=\{8\}, V_{4,1}=\{4\}$ and $V_{4,2}=\{7\}$, which each of $V_{1,1}, V_{2,1}, V_{3,1}$ and $V_{4,1}$ form a coalition with each of  $V_{1,2}, V_{2,2}, V_{3,2}$ and $V_{4,2}$, respectively. Therefore, we have a maximal $c$-partition of $G_{1}$ of order $8$ as follows.
		
		$\pi_{1}= \big\{ V_{1,1}=\{1\}, V_{1,2}=\{5\}, V_{2,1}=\{2\}, V_{2,2}=\{6\}, V_{3,1}=\{3\}, V_{3,2}=\{8\}, V_{4,1}=\{4\}, V_{4,2}=\{7\} \big\}.$

		\item[(ii)]  We show that there is no partition of order $8$ for $G_{4}$. Let $\pi'$ = $\big\{V'_1, V'_2, \ldots, V'_{8} \big\}$ be a $c$-partition of $G_4$. Since $|V(G_4)|=8$, so the only possible partition of eight sets is eight singleton sets. Since the graph $G_{4}$ has exactly two dominating sets of size $2$, without loss of generality, we may assume that there are only four singleton sets $V'_2$, $V'_4$, $V'_6$ and $V'_8$, in which $V_2'$ forms a coalition with $V_6'$, and $V_4'$ is in a coalition with $V_8'$. It follows that none of two singleton sets of the remaining singleton sets $V'_1$, $V'_3$, $V'_5$ and $V'_7$ form a coalition of size $2$. Therefore, there is no partition of order $8$. So, $C(G_{4})\leq 7$. Now we  show that there is a partition of order $7$.  Let $\pi''= \big\{V''_{1}=\{1,3,5,7\}, V''_{2}= \{2,6\}, V''_{3}= \{4,8\} \big\}$ be a domatic partition of a graph $G_{4}$, where $d(G_{4})=3$. We can partition each of the minimal dominating sets $V''_{2}=\{2,6\}$ and $V''_{3}= \{4,8\}$ into two sets such as $V''_{2,1}=\{2\}$, 
		$V''_{2,2}=\{6\}, V''_{3,1}=\{4\}$ and $V''_{3,2}=\{8\}$, in which $V''_{2,1}=\{2\}$ form a coalition with $V''_{2,2}=\{6\}$, and $V''_{3,1}=\{4\}$ is in a coalition with $V''_{3,2}=\{8\}$. Now we create a partition $\pi''_{1}$ of sets and put the sets $V''_{2,1}=\{2\}$, $V''_{2,2}=\{6\},V''_{3,1}=\{4\}$ and $V''_{3,2}=\{8\}$ in this partition. To obtain other sets of partition $\pi''_{1}$, let $W=\{1,5,7\}\subset V''_{1}$ be a minimal dominating set contained in $V''_{1}$. So, we can partition it into two non-dominating sets $W_{1}=\{1\}$ and $W_{2}=\{5,7\}$, add these two sets to $\pi''_{1}$. The set $T=\{3\}$ remains which is not a dominating set, else there are at least 4 disjoint dominating sets in $G_4$, a contradiction, because  $d(G_{4})=3$. The set $T$ forms a coalition with $W_{2}$, so  we can add $T$ to $\pi''_{1}$. Therefore, we have a maximal $c$-partition of $G_4$ of order $7$ as follows:
		
		$\pi''_{1} = \big\{W_{1}=\{1\}, W_{2}=\{5,7\},T=\{3\}, V''_{2,1}=\{2\},V''_{2,2}=\{6\}, V''_{3,1}=\{4\}, V''_{3,2}=\{8\} \big\}.$ 
		
		\item[(iii)] 
		Without loss of generality, we  compute $C(G)$ for the graph $G_2$.  First, we  show there is no partition of order $7$ for $G_2$. We  suppose $\psi = \big\{U_1, U_2, \ldots, U_{7} \big\}$. Since $|V(G_2)|=8$, so the only possible partition of $V(G_2)$  is six singleton sets and one doubleton set. Since $\gamma(G_2)=3$, no two singleton sets form a coalition and by Theorem \ref{atmost}, the doubleton set can be in a coalition with at most four singleton sets. Therefore,  $C(G_2)\leq 6$. In order to show that there is a partition of order 6, we make a domatic partition. Assume that $\psi'= \big\{U'_{1}=\{1,2,4,7\}, U'_{2}= \{3,5,6,8\} \big\}$ be a domatic partition of the graph $G_2$, where $d(G_{2})=2$. Since any partition of a non-singleton, minimal dominating set into two nonempty sets creates two non-dominating sets whose union forms a coalition, so we may assume that each of the minimal dominating sets $U''_{1}=\{1,4,7\} \subset U'_{1}$ and $U''_{2}= \{3,5,8\} \subset U'_{2}$ into two sets such as $U''_{1,1}=\{1,4\}, U''_{1,2}=\{7\}$, $U''_{2,1}=\{3,8\}$ and $U''_{2,2}=\{5\}$, in which $U''_{1,1}=\{1,4\}$ form a coalition with $U''_{1,2}=\{7\}$, and $U''_{2,1}=\{3,8\}$ is in a coalition with $U''_{2,2}=\{5\}$. Now we create a partition $\psi''$ of sets and put the sets $U''_{1,1}=\{1,4\}$, $U''_{2,1}=\{7\}$, $U''_{2,1}=\{3,8\}$ and $U''_{2,2}=\{5\}$ in this partition. The sets $U'''_{1}=\{2\} \subset U'_{1}$ and $U'''_{1}=\{2\} \subset U'''_{2}=\{6\}\subset U'_{2}$ remain which are not dominating sets, else there are at least 4 disjoint dominating sets in $G_2$, a contradiction. The set $U'''_{1}=\{2\}$ forms a coalition with $U''_{1,1}$ and $U'''_{2}=\{6\}$ is in coalition with  $U''_{2,1}=\{3,8\}$, and so we can add $U'''_{1}$ and $U'''_{2}$ to $\psi''$. Therefore, we have a $c$-partition of $G_2$ of order $6$ as follows:
		
		$\psi'' = \big\{U''_{1,1}=\{1,4\}, U''_{1,2}=\{7\}, U''_{2,1}=\{3,8\}, U''_{2,2}=\{5\},  U'''_{1}=\{2\}, U'''_{2}=\{6\}\big\}$.
		
		Therefore, we have the result.\qed 
	
	\end{enumerate}
	
	\end{proof}

\section{Coalition number of cubic graphs of order $10$ }

In this section, we study the coalition number of cubic graphs of order $10$.  In particular, we obtain the coalition number of the Petersen graph. 

\begin{figure}[H]
	\hglue2.5cm
	\includegraphics[width=11cm,height=4.89cm]{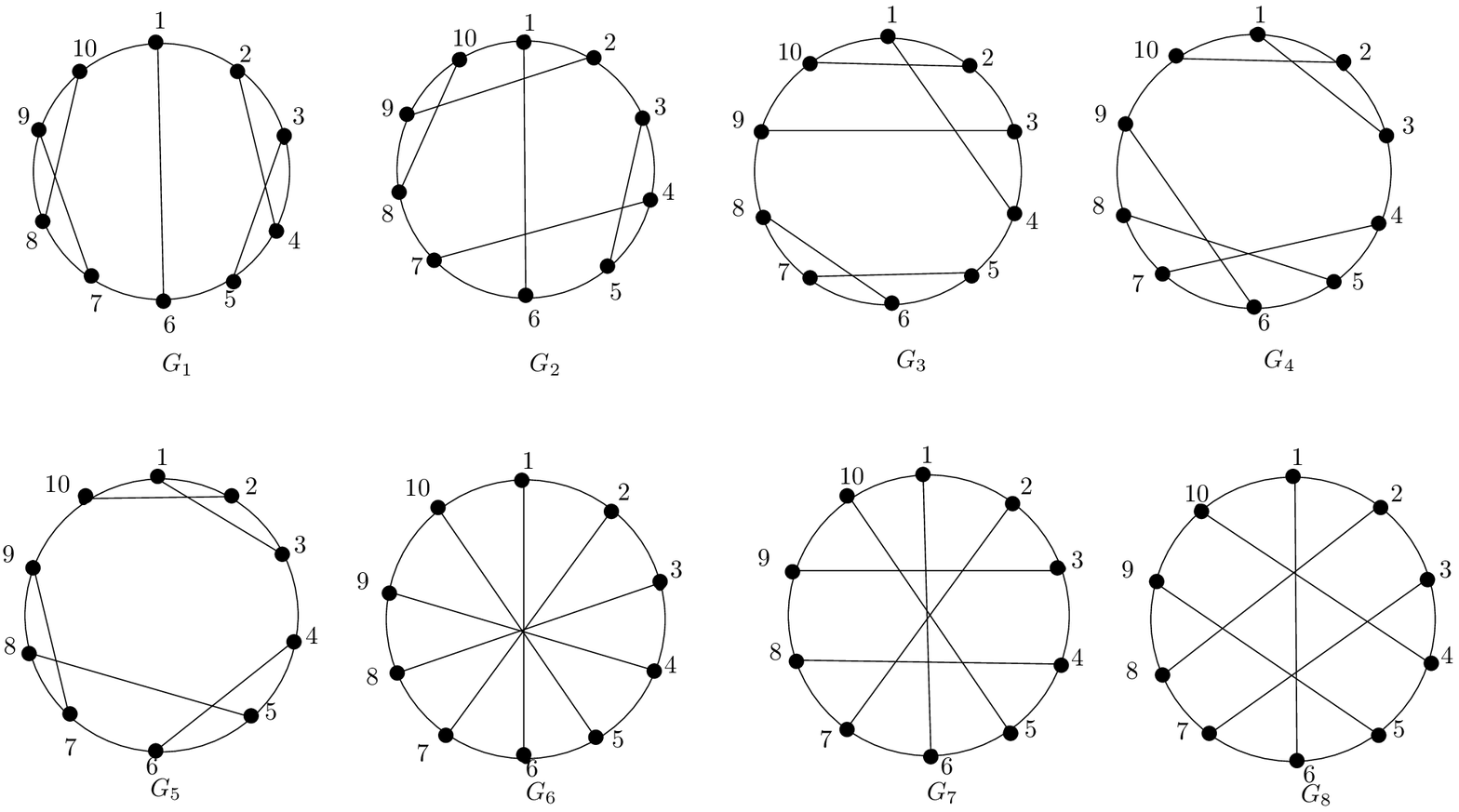}
	\vglue5pt
	\hglue2.5cm
	\includegraphics[width=11cm,height=4.99cm]{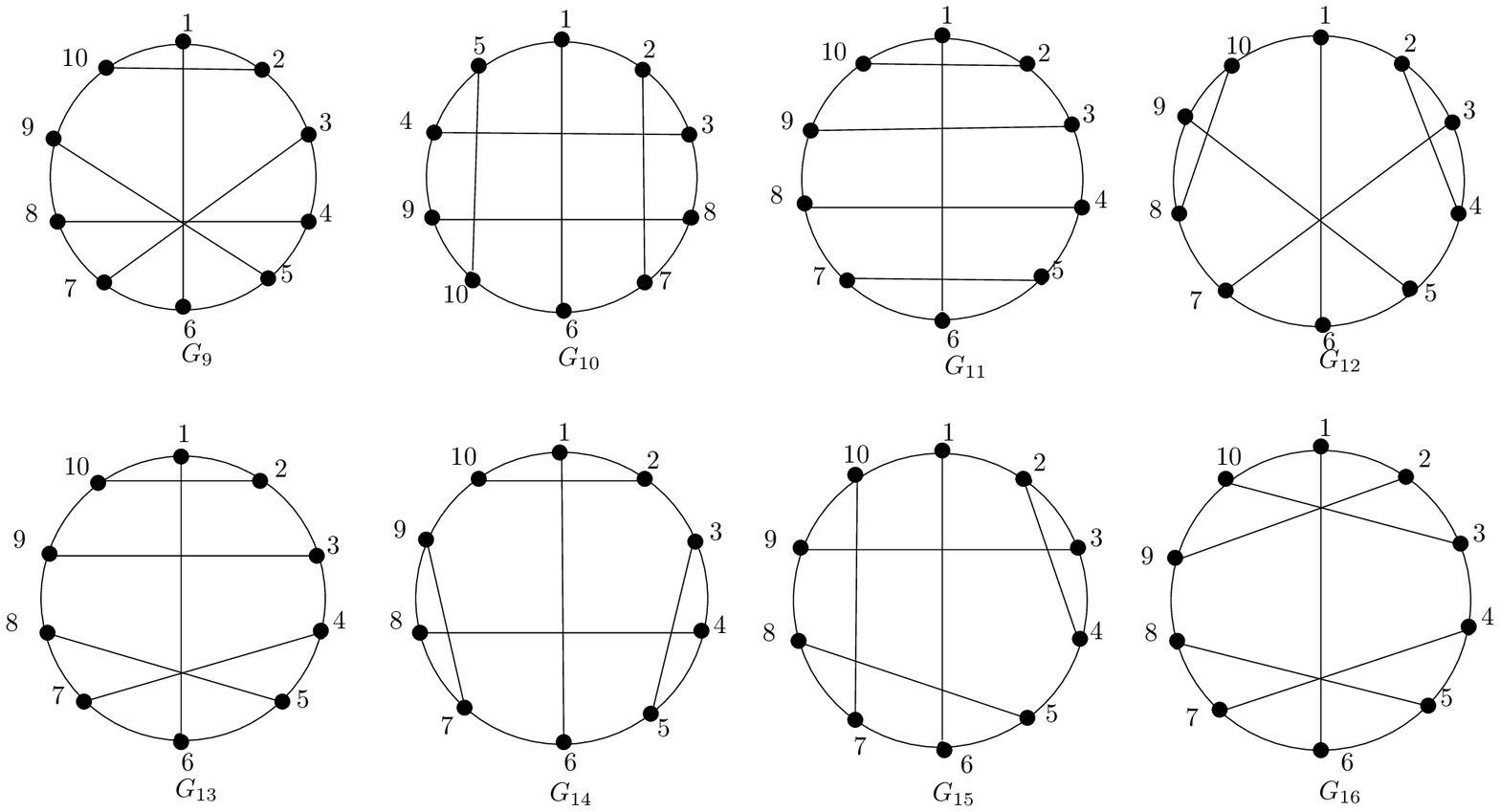}
	\hglue2.5cm
	\includegraphics[width=10.7cm,height=4.9cm]{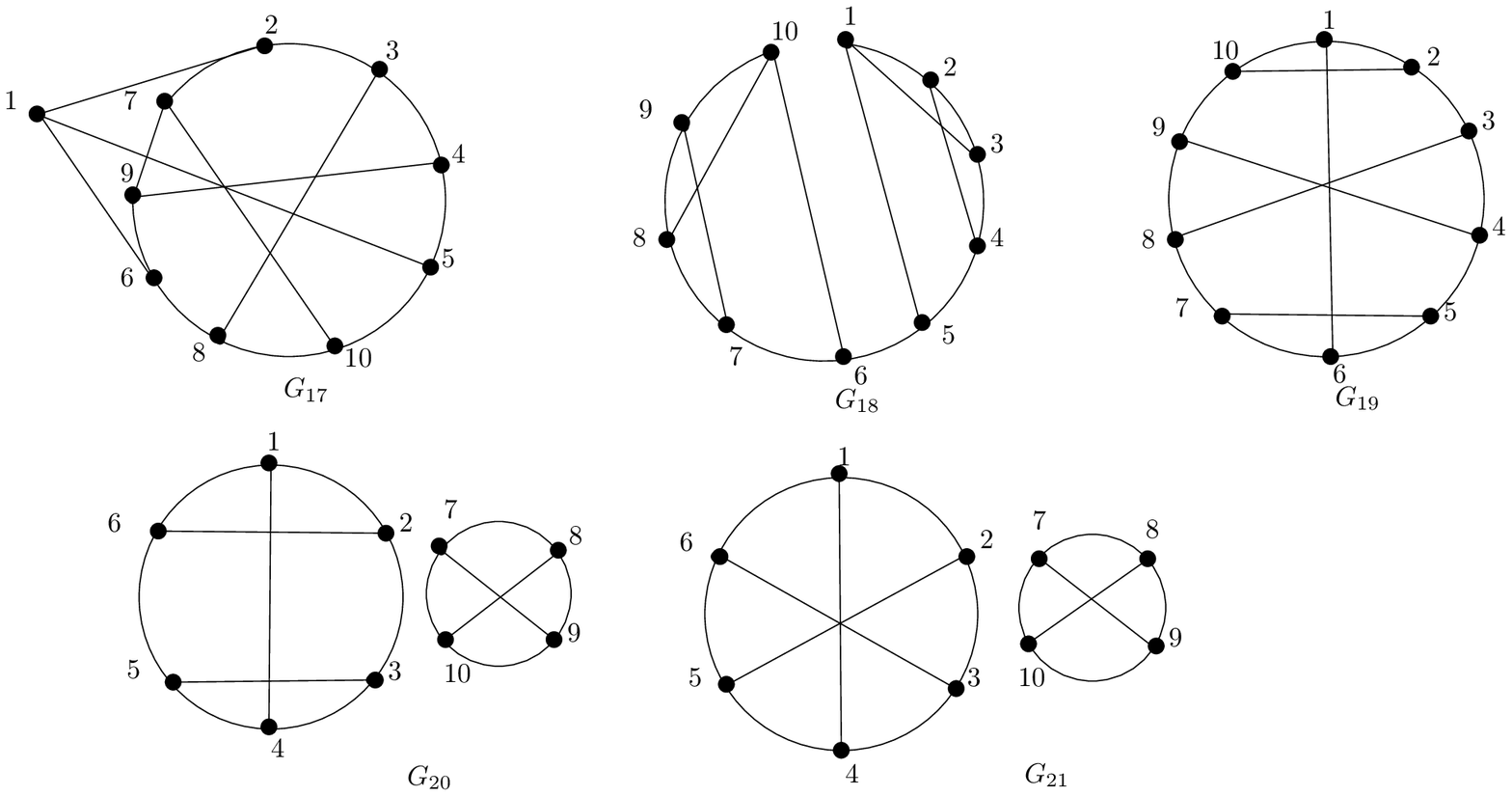}
	\hglue2.5cm
	\vglue-10pt \caption{\label{figure2} Cubic graphs of order $10$.}
\end{figure}

\medskip
There are exactly $21$ cubic graphs of order $10$ denoted by  $G_{1},G_{2},...,G_{21}$ in Figure \ref{figure2} (see \cite{10}). 
In particular, the graph $G_{17}$ is isomorphic to the Petersen graph $P$.
Now we state and prove the following theorem.

\begin{theorem}\label{8}
	For the cubic graphs $G_{1},G_{2},...,G_{21}$ of order $10$ (Figure \ref{figure2}) we have: 	
	\begin{enumerate} 
		\item[(i)] 
		$C(G_1)=8$. 
		\item[(ii)] 
		$C(G_i)=7$ for $i\in \{2,3,4,5,6,7,9,10,11,12,14,15,18,19,20,21\}$. 
		\item[(iii)]  
		$C(G_j)=6$ for $j\in \{8,13,16,17\}$.	 
	\end{enumerate} 
\end{theorem} 
\begin{proof} 
	By Theorems \ref{low} and \ref{up}, for the cubic graphs of order $10$, $5\leq C(G)\leq 9$. 
	\begin{enumerate}
		\item[(i)] 	We show that $ C(G_{1})\ne 9 $. Suppose that there exists a $c$-partition $\pi$ = $\big\{V_1, \ldots, V_{9} \big\}$ for the graph $G_{1}$ with nine classes. So, every set $V_k$ of $\pi$ is either a singleton dominating set of $G_{1}$, or is not a dominating set of $G$ but forms a coalition with another non-dominating set $V_l\in \pi$.
	  Since the graph $G_{1}$ has order $10$, exactly one class of $\pi$ must contains two vertices. Without loss of generality, let $|V_1|=2$ and $|V_j|=1$ for $2\leq j \leq9$. Let $S$ be a dominating set in the graph $G_{1}$. Since $|S|\ge 3$, no two singleton sets form a coalition. Now every set $V_j$ for $2\leq j \leq9$ must be in a coalition with $V_1$. Moreover, by Theorem \ref{atmost}, $V_1$ is in a coalition with at most $\Delta(G)+1=4$ sets. That means $V_1$ is a coalition partner of at most 4 singleton sets. Therefore, remain at least 4 singleton sets, neither of which can form a coalition with $V_1$, and it is a contradiction. So, we cannot create a partition of order $9$. Therefore, $C(G_{1})\leq 8$.

		Now we show that there is a partition of order $8$ for the graph $G_{1}$. We first construct a domatic partition. Let $\pi'= \big\{V_{1}=\{1,5,7\}, V_{2}= \{2,3,4,6,8,9,10\}\big\}$ be a domatic partition of a graph $G_{1}$, where $d(G_{1})=2$. We can partition the minimal dominating set $V_{1}=\{1,5,7\}$ into two sets $V_{1,1}=\{7\}$ and $V_{1,2}=\{1,5\}$, which form a coalition. Now we create a partition $\pi''$ of sets and put the sets $V_{1,1}=\{7\}$ and $V_{1,2}=\{1,5\}$ in this partition. To obtain other sets of the partition $\pi''$, let $V'_{2}=\{2,6,10\}\subset V_{2}$ be a minimal dominating set contained in $V_{2}$. We can partition it into two non-dominating sets $V'_{2,1}=\{2\}$ and $V'_{2,2}=\{6,10\}$, add these two sets to $\pi''$. The set $W=\{3,4,8,9\}$ remains which is not a dominating set, else there are at least 3 disjoint dominating sets in $G_{1}$, a contradiction, since $d(G_{1})=2$. The set $W$ forms a coalition with other sets, then we can add $W$ to $\pi''$. Thereby creating a $c$-partition of $G_{1}$ of order at least $5$, however, we can create a $c$-partition with order $8$ because there are two sets $W_{1}=\{3\}\subset W$ and $W_{2}=\{4\}\subset W$ form a coalition with $V'_{2,2}=\{6,10\}$, and also two sets $W_{3}=\{8\}\subset W$ and $W_{4}=\{9\}\subset W$ form a coalition with $V'_{1,2}=\{1,5\}$. Therefore, we can remove the set $W$ and add the sets $W_{1}=\{3\}$, $W_{2}=\{4\}$, $W_{3}=\{8\}$ and $W_{4}=\{9\}$ to $\pi''$. So we have a $c$-partition of $G_{1}$ of order $8$ as follows:\\
		$\pi'' = \big\{V_{1,1}=\{7\},V_{1,2}=\{1,5\},V'_{2,1}=\{2\}, V'_{2,2}=\{6,10\},W_{1}=\{3\},W_{2}=\{4\}, W_{3}=\{8\}, W_{4}=\{9\}\big\}$. 
		

		\item[(ii)] 
		Without loss of generality, we compute $C(G)$ for the graph $G_{2}$. Similar to Part (i) there is no $c$-partition of order $9$ for $G_{2}$. Now we show that there is no partition of order $8$ for $G_{2}$. Let $\pi_{1}$ = $\big\{V_1, V_2, \ldots, V_{8} \big\}$ be a $c$-partition of $G_2$. We consider two cases.   
		
		\nt {\bf Case 1}. There are exactly two sets of $\pi_{1}$ which consist of two vertices. Without loss of generality, let $|V_1|=|V_2|=2$ and $|V_j|=1$ for $3\leq j \leq 8$. Let $S$ be a dominating set in the graph $G_{2}$. Since $|S|\ge 3$, no two singleton sets form a coalition. Now every set $V_j$ for $3\leq j \leq 8$ must be in a coalition with at least one of $V_1$ or $V_2$. By Theorem \ref{atmost},
		each of $V_1$ and $V_2$ is in at most $\Delta(G)+1=4$ coalitions. Let $V_1$ and $V_2$ form a coalition, then each of $V_1$ and $V_2$ can be in at most three additional coalitions. Without loss of generality, suppose there are two collections such as $\pi_{1,1}$= $\big\{V_{1}, V_{3}, V_{4}, V_{5}\big\}\subset\pi_{1}$, which $V_{1}$ is in a coalition with each of singleton sets $V_{3}$, $V_{4}$ and $V_{5}$ and $\pi_{1,2}$= $\big\{V_{2}, V_{6}, V_{7}, V_{8}\big\}\subset\pi_{1}$, 
		which $V_{2}$ is in a coalition with each of singleton sets $V_{6}$, $V_{7}$ and $V_{8}$ such that $\pi_{1,1} \cup \pi_{1,2} = \pi_{1}$ and $\pi_{1,1} \cap \pi_{1,2} = \emptyset$. Note that  $V_1$ and $V_2$ form dominating sets of cardinality $3$ with singleton sets. All dominating sets of cardinality $3$ of  the graph $G_{2}$ are: \\
		$\{1, 3, 8\}, \{1, 4, 8\}, \{1, 4, 9\}, \{1, 4, 10\}, \{1, 5, 8\}, \{2, 5, 8\}, \{3, 6, 8\}, \{3, 6, 9\}, \{3, 6, 10\}, \\\{3, 7, 10\}, 
		\{4, 6, 9\}, \{5, 6, 9\}.$

		Now we show there are no two collections such as $\pi_{1,1}$ and $\pi_{1,2}$ such that $\pi_{1,1} \cup \pi_{1,2}=\pi_{1}$ and $\pi_{1,1} \cap \pi_{1,2} = \emptyset$. Since $V_1$ and $V_2$ are in a coalition, without loss of generality, we can suppose $V_1 = \{1,4\}$ and $V_2 =\{6,9\}$, and so $V_1$ is a coalition partner with each of $V_{3}=\{8\}$, $V_{4}=\{9\}$ and $V_{5}=\{10\}$, and also $V_2$ is in a coalition with each of $V_{6}=\{8\}$, $V_{7}=\{9\}$ and $V_{8}=\{10\}$. Therefore, $\pi_{1,1} \cup \pi_{1,2} \neq \pi_{1}$ and $\pi_{1,1} \cap \pi_{1,2} \neq \emptyset$, and it is a contradiction. So, there are no two collections such as $\pi_{1,1}$ and $\pi_{1,2}$ when $V_1$ and $V_2$ are coalition partners. Implying that, there is no a partition of order 8. Hence, $C(G_{2})\leq 7$. Similarly, an identical argument shows the same result holds when $V_1$ and $V_2$ form a coalition, and $V_1$ form a coalition with any three singleton sets of $\pi_{1}$ and $V_2$ form a coalition with each of the other three singleton sets of $\pi_{1}$.
		If $V_1$ and $V_2$ do not form a coalition, by Theorem \ref{atmost}, 
		each of $V_1$ and $V_2$ can be in at most four coalitions. Without loss of generality, let two collections such as $\pi_{1,1}$= $\big\{V_{1}, V_{3}, V_{4}\big\}\subset\pi_{1}$, which $V_{1}$ is in a coalition with each of singleton sets $V_{3}$, $V_{4}$ and $\pi_{1,2}$= $\big\{V_{2}, V_{5}, V_{6}, V_{7}, V_{8}\big\}\subset\pi_{1}$, 
		which $V_{2}$ is in a coalition with each of singleton sets $V_{5}$, $V_{6}$, $V_{7}$ and $V_{8}$ such that $\pi_{1,1} \cup \pi_{1,2} = \pi_{1}$ and $\pi_{1,1} \cap \pi_{1,2} = \emptyset$. Without loss of generality, assume that $V_1 = \{5,8\}$ is a coalition partner with each of $V_{3}=\{1\}$ and $V_{4}=\{2\}$. It follows that the remaining vertices such as $3,4,6,7,9,10$ must create the collection $\pi_{1,2}=\big\{V_{2}, V_{5}, V_{6}, V_{7}, V_{8}\big\}$, which $V_{2}$ is in a coalition with each of singleton sets $V_{5}$, $V_{6}$, $V_{7}$ and $V_{8}$. Without loss of generality, let $V_{2}=\{3,9\}$, $V_{5}=\{6\}$, $V_{6}=\{4\}$, $V_{7}=\{7\}$ and $V_{8}=\{10\}$. It can be seen, $V_{2}$ forms a coalition with $V_{5}$ and does not form a coalition with each of $V_{6}$, $V_{7}$ and $V_{8}$. Then, there is no the collection $\pi_{1,2}$. Implying that,  $\pi_{1,1} \cup \pi_{1,2} \neq \pi_{1}$ and $\pi_{1,1} \cap \pi_{1,2} \neq \emptyset$, and it is a contradiction. Therefore, there is no a partition of order $8$ and so $C(G_{2})\leq 7$. Similarly, an identical argument shows for any two sets of $\pi_{1}$ such as $V_1$ and $V_2$, which $V_1$ and $V_2$ do not form a coalition, and $V_1$ form a coalition with any two singleton sets of $\pi_{1}$ and $V_2$ form a coalition with each of the other four singleton sets of $\pi_{1}$ the same result holds. To complete the proof in the Case $1$, when $V_1$ and $V_2$ are
		not coalition partners, we also may assume that $\pi_{1,1}$= $\big\{V_{1}, V_{3}, V_{4}, V_{5}\big\}\subset\pi_{1}$, which $V_{1}$ is in a coalition with each of singleton sets $V_{3}$, $V_{4}$ and $V_{5}$ and $\pi_{1,2}$= $\big\{V_{2}, V_{6}, V_{7}, V_{8}\big\}\subset\pi_{1}$, 
		which $V_{2}$ is in a coalition with each of singleton sets $V_{6}$, $V_{7}$ and $V_{8}$, and then by  an identical argument the same result holds. Note that for notational convenience, let $|\pi_{1,1}|=|\pi_{1,2}|=4$.
		Moreover, an identical argument shows the same result holds when $V_1$ and $V_2$ do not form a coalition partner and we may assume that $|\pi_{1,1}|=|\pi_{1,2}|=5$, or $|\pi_{1,1}|=5$ and $|\pi_{1,2}|=4$.

		\medskip
		
		\nt {\bf Case 2.} We suppose that there are 7 singleton sets in the partition and exactly one set of $\pi_{1}$ must consists of three vertices. Without loss of generality, let $|V_1|= 3$ and $|V_j|=1$ for $2\leq j \leq 8$. Suppose $S$ is a dominating set in the graph $G_{2}$. Since $|S|\ge 3$, no two singleton sets form a coalition. Now every set $V_j$ for $2\leq j \leq 8$ must be in a coalition with $V_1$. By Theorem \ref{atmost}, $V_1$ is a coalition partner of at most 4 singleton sets. Therefore, remain at least 4 singleton sets, neither of which can form a coalition with $V_1$,
		 and it is a contradiction. It follows that we cannot create a partition of order $8$. Hence, $C(G_{2})\leq 7$. \\ 
		Now we show that there is a partition of order $7$ for the graph $G_{2}$.
		Let $\pi'_{1}= \big\{V'_{1}=\{2,5,8\}, V'_{2}= \{1,3,4,6,7,9,10\}\big\}$ be a domatic partition of a graph $G_{2}$, where $d(G_{2})=2$. Since any partition of a non-singleton, minimal dominating set into two nonempty sets creates two non-dominating sets whose union forms a coalition, so  we can partition the minimal dominating set $V'_{1}=\{2,5,8\}$ into two sets $V'_{1,1}=\{2\}$ and $V'_{1,2}=\{5,8\}$, which form a coalition. Now we create a partition $\pi''_{1}$ of sets and put the sets $V'_{1,1}=\{2\}$ and $V'_{1,2}=\{5,8\}$ in this partition. To obtain other sets of partition $\pi''_{1}$, let $V''_{2}=\{3,6,10\}\subset V'_{2}$ be a minimal dominating set contained in $V'_{2}$. Then we can partition it into two non-dominating sets $V''_{2,1}=\{6\}$ and $V''_{2,2}=\{3,10\}$, add these two sets to $\pi''_{1}$. The set $U=\{1,4,7,9\}$ remains which is a dominating set. Hence, let $U_{1}$=\{1,4,9\}$ \subset U $ be a minimal dominating set contained in $U$. Then we can partition it into two non-dominating sets $U_{1,1}=\{1\}$ and $U_{1,2}=\{4,9\}$, add these two sets to $\pi''_{1}$. Now remains the singleton set $U'=\{7\}$, which is not a dominating set. The set $U'$ forms a coalition with $V''_{2,2}=\{3,10\}$, then we can add $U'$ to $\pi''_{1}$. Thereby creating a $c$-partition of $G_{2}$ of order $7$. 
		And we have a $c$-partition of $G_{2}$ of order $8$ as follows.\\
		$\pi''_{1} = \big\{V'_{1,1}=\{2\},V'_{1,2}=\{5,8\},V''_{2,1}=\{6\}, V''_{2,2}=\{3,10\},U_{1,1}=\{1\},U_{1,2}=\{4,9\}, U'=\{7\}\big\}.$

		Note that using the same approach we can obtain the coalition number for other graphs in this part.
		
		\item[(iii)] Let us obtain the coalition number of  the graph $G_{17}$. From our previous discussions, similarly by the same argument, we can show there are no partitions of order $8$, $9$ and $10$ for the graph $G_{17}$. Therefore, now we show a partition of order $7$ does not exist for the graph $G_{17}$. We can suppose $\pi_{2} = \big\{V_1, V_2,\ldots, V_7\big\}$ is a partition with seven sets for $G_{17}$. We deduce the following cases. 
		
		\medskip 
		
		\nt {\bf Case 1.} We may assume that there are exactly three sets of $\pi_{2}$ must contain six vertices. Without loss of generality, let $|V_1|=|V_2|= |V_3|= 2$ and $|V_j|=1$ for $4\leq j \leq 7$. From our previous discussions, no two singleton sets can form a coalition. Now every set $V_j$ for $4\leq j \leq 7$ must be in a coalition with at least one of $V_1$, $V_2$ , or $V_3$. Let $V_1$ forms a coalition with each of $V_2$ and $V_3$, and $V_2$ forms a coalition with $V_3$. Then, by Theorem \ref{atmost}, each of $V_1$, $V_2$ and $V_3$ can be in at most two additional coalitions. Now we can assume that one of $V_1$, $V_2$, or $V_3$ is in a coalition with exactly two sets of $V_j$ for $4\leq j \leq 7$ and each of the two other sets can be in a coalition with one set of $V_j$ for $4\leq j \leq 7$. Without loss of generality, suppose there are three collections such as $\pi_{2,1}=\big\{V_{1}, V_{4}, V_{5}\big\}\subset\pi_{2}$, which $V_{1}$ is in a coalition with each of singleton sets $V_{4}$ and $V_{5}$, and $\pi_{2,2}= \big\{V_{2}, V_{6}\big\}\subset\pi_{2}$, 
		which $V_{2}$ is in a coalition with singleton set $V_{6}$, and $\pi_{2,3}= \big\{V_{3}, V_{7}\big\}\subset\pi_{2}$, 
		which $V_{3}$ is in a coalition with singleton set $V_{7}$ such that 
		 $\displaystyle \cup_{j=1}^{3}\pi_{2,j}= \pi_{2}$ and $\pi_{2,1}\cap \pi_{2,2}=\emptyset$, $\pi_{2,1}\cap \pi_{2,3}=\emptyset$ and 
		 $\pi_{2,2}\cap \pi_{2,3}=\emptyset$. Without loss of generality, suppose $\pi_{2,1}=\big\{V_{1}=\{1,3\}, V_{4}=\{2\}, V_{5}=\{5\}\big\}$, $\pi_{2,2}=\big\{V_{2}=\{4,10\}, V_{6}=\{8\}\big\}$ and $\pi_{2,3}=\big\{V_{3}=\{6,7\}, V_{7}=\{9\}\big\}$. It can be seen that, $V_1$ forms a coalition with each of $V_2$ and $V_3$, and $V_2$ forms a coalition with $V_3$, however, $V_{1}$ does not form a coalition with each of singleton sets $V_{4}$ and $V_{5}$, and $V_{2}$ does not form a coalition with $V_{6}$, and also $V_{2}$ is not in a coalition with singleton set $V_{7}$. Therefore, there are no three collections such as $\pi_{2,1}$, $\pi_{2,2}$ and $\pi_{2,3}$ and it is a contradiction. Implying that, there is no  a partition of order $7$ and then $C(G_{17})\leq 6$. Similarly, we can show when $V_1$, $V_2$ and $V_3$ be any three sets of vertices such that $|V_1|=|V_2|= |V_3|= 2$, and $V_1$ forms a coalition with each of $V_2$ and $V_3$, and $V_2$ forms a coalition with $V_3$, the same result holds. Note that if $V_1$ and $V_2$ do not form a coalition partner and we may assume 
		$|\pi_{2,1}|=|\pi_{2,2}|= |\pi_{2,3}|=3$, or $|\pi_{2,1}|=|\pi_{2,2}|=3$ and $|\pi_{2,3}|=2$, and then the same result holds.
		
		With an identical argument we have the same result, when at least one pair of sets $V_1$, $V_2$ or $V_3$ does not form a coalition. Therefore, $C(G_{17})\leq 6$.

		\medskip 
		\nt {\bf Case 2.} We may assume that there are $5$ singleton sets and two sets of $\pi_{2}$ must contain three and two vertices. Without loss of generality, let $|V_1|=3$, $|V_2|=2$ and $|V_j|=5$ for $3\leq j \leq 7$. Since $G_{2}$ has no dominating set with less than three vertices, then no two singleton sets are in a coalition. Now every set $V_j$ for $3 \leq j\leq7$ must be in a coalition with at least one of $V_1$ or $V_2$. Theorem \ref{atmost} implies that,
		each of $V_1$ and $V_2$ is in a coalition with at most $\Delta(G)+1=4$ sets. Let $V_1$ and $V_2$ form a coalition, then each of $V_1$ and $V_2$ can be in at most three additional coalitions. Without loss of generality, suppose there are two collections such as $\pi'_{2,1}= \big\{V_{1}, V_{3}, V_{4}, V_{5}\big\}\subset\pi_{2}$, which $V_{1}$ is in a coalition with each of singleton sets $V_{3}$,$V_{4}$ and $V_{5}$, and $\pi'_{2,2}=\big\{V_{2}, V_{6},V_{7}\big\}\subset\pi_{2}$, which $V_{2}$ is in a coalition with each of singleton sets $V_{6}$ and $V_{7}$ such that $\pi'_{2,1} \cup \pi'_{2,2} = \pi_{2}$ and $\pi'_{2,1}\cap \pi'_{2,2} =\emptyset$. Without loss of generality, let $\pi'_{2,1}=\big\{V_{1}=\{1,3,4\}, V_{3}=\{7\}, V_{4}=\{10\}, V_{5}=\{8\}\big\}$, and $\pi'_{2,2}=\big\{V_{2}=\{2,5\}, V_{6}=\{6\}, V_{7}=\{9\}\big\}$. It can be seen, $V_1$ and $V_2$ form a coalition, however, $V_{1}$ and $V_{2}$ do not form a coalition with $V_{5}$ and $V_{7}$, respectively. Then, there is at least one singleton set in $\pi'_{2,1}$, which does not form coalition with $V_1$, and also there is at least one singleton set in $\pi'_{2,2}$, which does not form a coalition with $V_2$. Therefore, there are no two collections such as $\pi'_{2,1}$, $\pi'_{2,2}$ and it is a contradiction. Implying that, there is no a partition of order $7$ and then $C(G_{17})\leq 6$. Similarly, we can show when $V_1$ and $V_2$ be any two sets of vertices such that $|V_1|=3$ and $|V_2|=2$, and $V_1$ forms a coalition with $V_2$, the same result holds. Hence, there is no a collection such as $\pi'_{2,1}$ and it is a contradiction. Implying that, there is no a partition of order $7$ and then $C(G_{17})\leq 6$. Note that if $V_1$ and $V_2$ form a coalition partner and we assume that $|\pi'_{2,1}|=3$ and $|\pi'_{2,2}|=4$, or $|\pi'_{2,1}|=|\pi'_{2,2}|=4$, then an identical argument proves the same result holds.
		
		If $V_1$ and $V_2$ do not form a coalition, then each of $V_1$ and $V_2$ is in coalition with at most $\Delta(G)+1=4$ sets. Without loss of generality, assume $\pi'_{2,1}= \big\{V_{1}, V_{3}, V_{4}, V_{5},V_{6}\big\}\subset\pi_{2}$, which $V_{1}$ is in a coalition with each of singleton sets $V_{3}$, $V_{4}$, $V_{5}$ and $V_{6}$, and $\pi'_{2,2}= \big\{V_{2}, V_{7}\big\}\subset\pi_{2}$, which $V_{2}$ is in a coalition with singleton set $V_{7}$ such that $\pi'_{2,1} \cup \pi'_{2,2} = \pi_{2}$ and $\pi'_{2,1}\cap \pi'_{2,2} =\emptyset$. Without loss of generality, let $V_2=\{3,5\}$, $V_7 =\{9\}$. According to the all dominating sets of cardinality $4$, which are listed in \cite{1}, may be seen that every pair of dominating sets of cardinality 4, which have three vertices in common, either whose union of a dominating set of cardinality $3$ and a singleton set, or whose union of a non-dominating set of cardinality $3$ and a singleton set. Now we consider the remaining vertices $1,2,4,6,7,8,10$. Without loss of generality, assume $\pi'_{2,1}= \big\{V_{1}=\{1,2,4\}, V_{3}=\{6\}, V_{4}=\{7\}, V_{5}=\{8\}, V_{6}=\{10\}\big\}$. It can be seen, $V_{1}$ forms a coalition with each of $V_{5}$ and $V_{6}$, and does not form a coalition with each of $V_{3}$ and $V_{4}$. Therefore, there is no  collection such as $\pi'_{2,1}$ and it is a contradiction. Similarly, when $V_{1}$ is any non-dominating set of cardinality $3$, we cannot create 4 coalitions of order $4$, which have three vertices in common by remaining vertices $1,2,4,6,7,8,10$. Hence, we cannot create a partition of order $7$. Implying that, $C(G_{17})\leq 6$. Note that an identical argument shows if $V_2\cup V_7$ is any dominating set of cardinality $3$, then the same result holds. Finally, to complete our argument when $V_1$ and $V_2$ do not form a coalition partner, we assume following subcases, and then the same result holds for the following subcases.
		
		\nt {\bf Subcase 2.1} $|\pi'_{2,1}|=|\pi'_{2,2}|=5$.
		
		\nt {\bf Subcase 2.2}  $|\pi'_{2,1}|= 5$ and $|\pi'_{2,2}|=4$, or $|\pi'_{2,1}|= 4$ and $|\pi'_{2,2}|=5$.
		
		\nt {\bf Subcase 2.3} $|\pi'_{2,1}|= 5$ and  $|\pi'_{2,2}|=3$, or $|\pi'_{2,1}|= 3$ and  $|\pi'_{2,2}|=5$. 
		
		\nt {\bf Subcase 2.4} $|\pi'_{2,1}|= 2$,  $|\pi'_{2,2}|=5$. 
		
		\nt {\bf Subcase 2.5}  $|\pi'_{2,1}|= 4$,  $|\pi'_{2,2}|=4$.
		
		\nt {\bf Subcase 2.6}  $|\pi'_{2,1}|= 4$ and $|\pi'_{2,2}|=3$, or $|\pi'_{2,1}|= 3$ and $|\pi'_{2,2}|=4$.

		\medskip
		\nt {\bf Case 3.} We can suppose there are 6 singleton sets in the partition and exactly one set of $\pi_{2}$ must consists of four vertices. Without loss of generality, let $|V_1|= 4$ and $|V_j|=1$ for $2\leq j \leq 7$. Since $G_{17}$ has no dominating set with less than three vertices, no two singleton sets can be in a coalition. Now every set $V_j$ for $2\leq j \leq 7$ must be in a coalition with $V_1$. By Theorem \ref{atmost}, $V_1$ is in a coalition with at most $\Delta(G)+1=4$ sets. That means $V_1$ is a coalition partner of at most 4 singleton sets. Therefore, remain at least three singleton set, neither of
		which can form a coalition with $V_1$, and it is a contradiction. Then, we cannot create a partition of order $7$. Hence, $C(G_{17})\leq 6$. 
		
		To complete the proof, we create a $c$-partition of order $6$  for the graph $G_{17}$ and we show that this partition is the maximal partition. We first create a domatic partition.
		We can suppose $\pi''_{2}= \big\{V''_{1}=\{1,7,3\}, V''_{2}= \{2,4,5,6,8,9,10\}\big\}$ be a domatic partition of a graph $G_{17}$, where $d(G_{17})=2$. Since any partition of a non-singleton, minimal dominating set into two nonempty sets creates two non-dominating sets whose union forms a coalition, so we can partition the minimal dominating set $V_{1}=\{1,7,3\}$ into two sets $V''_{1,1}=\{1\}$ and $V''_{1,2}=\{7,3\}$, which form a coalition. Now we create a partition $\pi'''_{2}$ of sets and put the sets $V''_{1,1}=\{1\}$ and $V''_{1,2}=\{7,3\}$ in this partition. To obtain other sets of partition $\pi'''_{2}$, let $V'''_{2}=\{2,8,4\}\subset V''_{2}$ be a minimal dominating set contained in $V''_{2}$. Then, we can partition it into two non-dominating sets $V'''_{2,1}=\{2\}$ and $V'''_{2,2}=\{8,4\}$, add these two sets to $\pi'''_{2}$. The set $T=\{5,6,9,10\}$ remains which is not a dominating set, else there are at least 3 disjoint dominating sets in $G$, a contradiction, because  $d(G_{17})=2$. The set $T$ forms a coalition with other sets, so  we can add $T$ to $\pi'''_{2}$. Thereby creating a $c$-partition of $G_{17}$ of order at least $5$, however, we can create a $c$-partition with order $6$ because there are two sets $T_{1}=\{5,6\}\subset T$ and $T_{2}=\{9,10\}\subset T$ form a coalition with $V'''_{2,1}=\{2\}$. Therefore, we can remove the set $T$ and add the sets $T_{1}=\{5,6\}$ and $T_{2}=\{9,10\}$ to $\pi'''_{2}$. And finally we have a $c$-partition of $G $ of order $6$ as follows:\\
		
		$\pi'''_{2} = \big\{V''_{1,1}=\{1\},V''_{1,2}=\{7,3\},V'''_{2,1}=\{2\}, V'''_{2,2}=\{8,4\},T_{1}=\{5,6\},T_{2}=\{9,10\} \big\}.$ 
		
		We observe that  $\pi'''_{2}$ is a $c$-partition with maximum order. As we know, $C(G_{17})\ge 5$. Moreover, we have $ C(G_{17})\leq 6$. Therefore, $C(G_{17})= 6$. Using the argument used to obtain the coalition number of the graph $G_{17}$, we can obtain the coalition number for other graphs belong to this part. \qed
	\end{enumerate} 
	
\end{proof}

\section{Conclusion}

Study of the coalition number of regular graphs is a main subject in the coalition theory. In this paper, we have determined the coalition number of the cubic graphs of order at most $10$. We observed that the coalition number of all these graph  is $6$, $7$, or $8$. We think that this is true for all cubic graphs of order at least 6.  This raises us to propose the following open problem.  

\begin{problem}

Is it true that for any cubic graphs $G$ of order at least  $6$, $C(G)\in\{6,7,8\}?$. 
\end{problem}

Recently, total coalition has introduced and investigated in \cite{TC}. Two disjoint sets $V_1, V_2\subseteq V$ are called a total coalition in $G$, if neither  $V_1$ and $V_2$ is   a total dominating set of $G$ but $V_1\cup V_2$ is a total dominating set.  A total coalition partition of $G$ is a vertex partition $\pi=\{V_1,V_2,\ldots, V_k\}$ such that no set  of $\pi$ is  a total dominating set but  each set $V_i\in \pi$ forms a total coalition with another set $V_j\in \pi$.  The maximum cardinality of a total coalition partition of $G$ is called the total coalition number of $G$, denoted by $TC(G)$.   
It is natural to study the total coalition number of regular graphs, especially, cubic graphs. 
So we close the paper with the following open problems.

\begin{problem}  
What is the total coalition number of the cubic graphs of order more than four? 

\end{problem} 

\begin{problem}
	What are the sharp bounds for the total coalition number of $k$-regular graphs?
	\end{problem}

\medskip


\end{document}